\documentclass[letterpaper, 10 pt, conference]{ieeeconf}  % Comment this line out

\usepackage[hidelinks]{hyperref}
\hypersetup{
    colorlinks=true,
    linkcolor=black,
    filecolor=magenta,      
    urlcolor=blue,
}
\usepackage{amsmath}
\usepackage{graphicx}
\usepackage{amssymb}
\usepackage{colortbl}
\usepackage{arydshln}
\usepackage{stfloats}
\usepackage{lipsum}
\usepackage[vlined,boxed,ruled,linesnumbered]{algorithm2e}
\SetKwProg{Fn}{Function}{}{}
\usepackage{algpseudocode}
\usepackage{multirow}
\usepackage{extarrows}
\usepackage{mathrsfs}
\usepackage{upgreek}
\newtheorem{theorem}{Theorem}
\newtheorem{remark}{Remark}

\newtheorem{definition}{Definition}
\newtheorem{lemma}{Lemma}

\definecolor{mygray}{gray}{0.8}

\IEEEoverridecommandlockouts                              % This command is only
\title{\LARGE \bf Necessary and Sufficient Conditions for Stability of Discrete-Time Switched Linear Systems with Ranged Dwell Time}

\author{Weiming Xiang,~\IEEEmembership{Senior Member, IEEE}% <-this % stops a space
\thanks{}% <-this % stops a space
\thanks{W. Xiang is with the School of Computer and Cyber Sciences, Augusta University, Augusta GA 30912 USA. Email:
        {\tt\small wxiang@augusta.edu}}%%
}

\begin{document}

\maketitle
\thispagestyle{empty}
\pagestyle{empty}

%%%%%%%%%%%%%%%%%%%%%%%%%%%%%%%%%%%%%%%%%%%%%%%%%%%%%%%%%%%%%%%%%%%%%%%%%%%%%%%%
\begin{abstract}
\boldmath
This paper deals with the stability analysis problem of discrete-time switched linear systems with ranged dwell time. A novel concept called $L$-switching-cycle is proposed, which contains sequences of multiple activation cycles satisfying the prescribed ranged dwell time constraint. Based on $L$-switching-cycle, two sufficient conditions are proposed to ensure the global uniform asymptotic stability of discrete-time switched linear systems. It is noted that two conditions are equivalent in stability analysis with the same $L$-switching-cycle. These two sufficient conditions can be viewed as generalizations of the clock-dependent Lyapunov and multiple Lyapunov function methods, respectively. Furthermore, it has been proven that the proposed $L$-switching-cycle can eventually achieve the nonconservativeness in stability analysis as long as a sufficiently long $L$-switching-cycle is adopted. A numerical example is provided to illustrate our theoretical results.
\end{abstract}

%%%%%%%%%%%%%%%%%%%%%%%%%%%%%%%%%%%%%%%%%%%%%%%%%%%%%%%%%%%%%%%%%%%%%%%%%%%%%%%%
\section{Introduction}
Switched systems have emerged as an important subclass of hybrid systems which consist of a finite number of subsystems described by differential or difference equations and a switching rule orchestrating the activation of those subsystems along with time. The motivation of studying switched systems comes from the fact that switched systems can be used to efficiently model many practical systems such as cyber-physical systems which are inherently with multiple working modes coordinated by logic rules, and to take the
electronic circuits field for example, DC/DC converters \cite{ma2004bifurcation}, oscillators \cite{torikai1998synchronization}, and chaos generators \cite{mitsubori1997dependent}. In addition, many control system problems can be incorporated in the framework of switched systems, e.g., event-triggered control systems \cite{heemels2012periodic,xiang2017event}, sampled-data systems \cite{hauroigne2011switched}, networked control systems \cite{zhang2012network}, time-delay systems \cite{hetel2008equivalence,wang2019necessary,xiang2018parameter}, etc.  

Among the large variety of problems studied in theory and encountered in practice, stability analysis of switched systems is a fundamental problem, which attracts considerable research attention in decades, readers may refer to \cite{liberzon2003switching,lin2009stability,goedel2012hybrid}, and the references cited therein. Lyapunov functions have been proved to be a powerful tool to deal with stability analysis for switched systems under either arbitrary switching or restricted switching. Under arbitrary switching, classical common Lyapunov function and switched Lyapunov function methods are proposed in \cite{daafouz2002stability}. Recently, a multiple-step method is proposed in \cite{xiang2017robust} which extends the construction of Lyapunov functions for each subsystem to the combination of subsystems in a finite-time window. For restricted switching, dwell time is usually used to characterize the switching rate of a switched system. There are a variety of Lyapunov function results developed for switched systems under dwell time switching. Multiple Lyapunov functions are constructed for subsystems, e.g., stability criteria in \cite{geromel2006stability,hespanha1999stability,zhai2001stability}. Clock-dependent Lyapunov function methods are proposed for discrete-time switched systems in \cite{briat2014convex}. Moreover, it has been demonstrated in \cite{briat2014convex} that multiple Lyapunov function methods and clock-dependent Lyapunov function methods are equivalent in stability analysis. In \cite{lee2006uniform}, a finite-path-dependent quadratic Lyapunov function is proposed for the stability of switched systems with strongly connected switching path constraint. Recently, a class of homogeneous polynomial Lyapunov functions with a sufficiently
large degree are constructed in \cite{xiang2016necessary} to perform nonconservative stability analysis for continuous-time switched linear systems.

In this paper, we aim to develop nonconservative stability criteria for stability of discrete-time switched linear systems with a ranged dwell time constraint. A novel concept called $L$-switching-cycle is introduced to improve both the multiple Lyapunov function approach and the clock-dependent Lyapunov function approach. First, two equivalent sufficient stability criteria are proposed in the framework of $L$-switching-cycle, which are the generalizations of results of multiple Lyapunov function approach and clock-dependent Lyapunov function approach, respectively. Moreover, it is proved that the $L$-switching-cycle approach can achieve nonconservative stability analysis as long as the length of $L$ is sufficiently large. At last, a comparison is made between the developed $L$-switching-cycle methods and several well-known stability criteria. It is noted that the $L$-switching cycle method can cover various existing stability criteria.  

%The rest of this paper is organized as follows. Preliminaries and problem formulation are given in Section II. The main contribution,  stability analysis based on $L$-switching-cycle, is presented in Section III. One numerical example is presented to validate our results in Section IV, and conclusions are given in Section V.

\emph{Notations:} $\mathbb{N}$ denotes the set of natural numbers, $\mathbb{R}$ represents the field of real numbers. $\mathbb{R}^{+}$ stands for the set of nonnegative real numbers, and $\mathbb{R}^{n}$ is the vector space of all $n$-tuples
of real numbers, $\mathbb{R}^{n \times n}$ is the space of $n\times n$ matrices with real entries. The set of $n \times n$ positive definite symmetric matrices is denoted by $\mathbb{S}^{n}_{\succ 0}$. %$\left\|\cdot\right\|$ stands for Euclidean norm. 
Given a matrix $X$, the notation $X \succ 0$ means $X$ is real symmetric and positive definite. Given another matrix $Y$, $X \succ Y$ means that
$X - Y \succ 0$. $X^{\top}$ denotes the transpose of $X$. A continuous function $\alpha:\mathbb{R}^{+} \to \mathbb{R}^{+}$ is a class $\mathcal{K}$ function if it is strictly increasing and $\alpha(0) = 0$. Moreover, a function $\beta:\mathbb{R}^{+} \times  \mathbb{R}^{+} \to \mathbb{R}^{+}$ is a class $\mathcal{KL}$ function if, for each fixed $s$, function $\beta(r,s)$ is a class $\mathcal{K}$ function with respect to $r$ and, for each fixed $r$, function $\beta(r,s)$ is decreasing with respect to $s$ and $\beta(r,s) \to 0$ as $s \to \infty$. For two integers
$n$ and $m$, $n \le m$, we define $\mathcal{I}[n, m]  \triangleq \{n,n+1,\ldots,m\}$.

\section{Preliminaries and Problem Formulation}
In this paper, we consider a class of discrete-time switched linear systems in the form of 
\begin{align} \label{eq:system}
    x(k+1) = A_{\sigma(k)}x(k)
\end{align}
where $x(k) \in \mathbb{R}^{n_x}$ is system state, $A_i \in \mathbb{R}^{n_x \times n_x}$ are system matrices. $\sigma:\mathbb{N} \to \mathcal{I}[1,N]$ is a piecewise constant function of time, called switching law or switching signal,
which takes value in a finite index set $\mathcal{I}[1,N] \triangleq \{1, 2, \ldots , N\}$, where $N > 0$ is the number of subsystems. Let the discontinuity points of $\sigma(k)$ be denoted by $k_n$, and let $k_0$ stand for the initial time by convention, the switching sequence can be described as $\mathcal{S} \triangleq \{k_n\}_{n \in \mathbb{N}}$. Calling $\mathcal{D}_{[\tau_{\min},\tau_{\max}]}$ the
set of all switching policies with a ranged  dwell time $[\tau_{\min},\tau_{\max}]$ with $1\le \tau_{\min}\le \tau_{\max} < \infty$, that is the set of all $\sigma(k)$ for which the time interval between successive discontinuities of $\sigma(k)$
satisfies $k_{n+1} - k_n \in [\tau_{\min},\tau_{\max}]$ , $\forall n \in \mathbb{N}$. We also define $d = \tau_{\max} -\tau_{\min} + 1$ in the rest of this paper. 

\begin{definition} \cite{geromel2006stability,khalil2002nonlinear}
The equilibrium $x = 0$ of system (\ref{eq:system}) is globally uniformly asymptotically stable (GUAS) under
switching signal $\sigma(k)$ if, for any initial condition $x(0)$, there
exists a class $\mathcal{KL}$ function $\beta$  such that the solution of system
(\ref{eq:system}) satisfies $\left\|x(k)\right\| \le \beta(\left\|x(0)\right\|,k)$, $\forall k \in \mathbb{N}$. 
\end{definition}

In terms of stability analysis for switched systems, the following two results are recalled. 

\begin{lemma} \cite{li2019dwell,briat2013convex}\label{lemma1}
Switched system (\ref{eq:system}) is GUAS with any switching signals $\sigma(k) \in \mathcal{D}_{[\tau_{\min},\tau_{\max}]}$ if there exist symmetric matrix sequences $P_{h}:\mathcal{I}[0,\tau_{\max}] \to \mathbb{S}^{n_x}_{\succ 0}$, $h \in \mathcal{I}[1, N]$ such that
\begin{align}
         A_{i}^{\top}P_{i}(k+1)A_{i}-P_{i}(k) &\prec 0\label{eq:lemma1_1}
         \\
         P_{i}(0) - P_{j}(\tau) &\prec 0 \label{eq:lemma1_2}
    \end{align}
hold for all $k \in [0,\tau_{\max}-1]$, $\tau \in [\tau_{\min},\tau_{\max}]$, $i,j \in \mathcal{I}[1, N]$. 
\end{lemma}

\begin{lemma} \cite{li2019dwell,briat2013convex} \label{lemma2}
Switched system (\ref{eq:system}) is GUAS with any switching signals $\sigma(k) \in \mathcal{D}_{[\tau_{\min},\tau_{\max}]}$ if  there exist matrices $P_h \in \mathbb{S}_{\succ 0}^{n_x}$, $h \in \mathcal{I}[1, N]$ such that
    \begin{align}
         (A_i^{\tau})^{\top}{P_{i}}(A_i^{\tau})-P_{j} \prec 0 \label{eq:lemma2_1}
    \end{align}
   hold for $\tau \in [\tau_{\min},\tau_{\max}]$, $i,j \in \mathcal{I}[1, N]$.
\end{lemma}

\begin{remark}
As observed in \cite{li2019dwell,briat2013convex}, the above two lemmas are essentially equivalent. Each stability criterion has its own advantages, e.g., conditions in Lemma \ref{lemma1} are convex in system matrices $A_i$ so that the results can be easily generalized to other related problems such as $\ell_2$-gain analysis, controller design, etc. On the other hand, the condition in Lemma \ref{lemma2} is with less computational cost as fewer decision variables and constraints are involved than Lemma \ref{lemma1}. However,  both lemmas are sufficient conditions for stability analysis. In this work, we aim to further extend these two results to achieve less conservative or even nonconservative stability analysis results. 
\end{remark}

Inspired by the multiple-step method proposed in \cite{xiang2017robust} for arbitrary switching, we propose a novel concept called $L$-switching-cycle to handle ranged dwell time switching, and it will be playing a critical role in the rest of this paper. 

\begin{definition}\label{def2}
Given an $L \in \mathbb{N}$,  an $L$-switching-cycle sequence is a sub-sequence $\mathcal{S}_L \triangleq \{k_{nL}\}_{n \in \mathbb{N}}$ stemmed from switching sequence $\mathcal{S}$.  Considering time interval $[k_{nL}, k_{(n+1)L})$, there are $N(N-1)^{L-1}$ combinations of subsystem invocations in $[k_{nL}, k_{(n+1)L})$ and moreover, due to dwell time $\tau \in [\tau_{\min},\tau_{\max}]$, there are $N(N-1)^{L-1}d^{L}$ sequences of trajectories of $\sigma(k)$.  The sequences of indices of subsystems in the $h$th combination and subsystem activation time are expressed by
\begin{align}
    \mathcal{L}_h& = \{i_{h,1},i_{h,2},\ldots,i_{h,L}\}
    \\
    \Gamma_h&  = \{\tau_{h,1},\tau_{h,2},\ldots,\tau_{h,L}\},~\tau_{h,\ell} \in [\tau_{\min},\tau_{\max}]
\end{align}
where $h \in \mathcal{I}[1, N(N-1)^{L-1}d^{L}]$, $i_{h,1},\ldots,i_{h,L} \in \mathcal{I}[1,N]$ and $i_{h,\ell} \ne i_{h,\ell+1}$, $\ell \in \mathcal{I}[1,L-1]$. The length of activation time of $\mathcal{L}_h$ is denoted by $T_h = \sum\nolimits_{i=1}^{L}{\tau_{h,i}}$. 
\end{definition}
\begin{remark}
If $i_{h,\ell} = i_{h,\ell+1}$ is allowed in $\mathcal{L}_h$, the special case with dwell time $\tau_{\max}=\tau_{\min}=1$ implies arbitrary switching, the $L$-switching-cycle sequence is reduced to a  multiple-step sequence with a length of $L$ as proposed in \cite{xiang2017robust}. Thus, the $L$-switching-cycle sequence can be viewed as a generalization of the concept of  multiple-step sequence, from arbitrary switching with $\tau_{\max}=\tau_{\min}=1$ to ranged dwell-time constrained switching $\tau_{\max} \ge \tau_{\min} \ge 1$. 
\end{remark}

\section{Main Results}

\subsection{Stability Criteria}

As presented in the following theorem, two stability criteria are developed based on $L$-switching-cycle as defined in Definition \ref{def2}, which is the main contribution of this paper. 

\begin{theorem} \label{thm1}
For switched linear system (\ref{eq:system}) and given an $L \in \mathbb{N}$, the following two statements are equivalent:

\begin{enumerate}
    \item [(\emph{a})] There exist symmetric matrix sequences $P_{h}:\mathcal{I}[0,T_{h}] \to \mathbb{S}^{n_x}_{\succ 0}$, $h \in \mathcal{I}[1, N(N-1)^{L-1}d^{L}]$ such that the following conditions
    \begin{align}
         A_{i_{h,\ell}}^{\top}P_{h}(k+1)A_{i_{h,\ell}}-P_{h}(k) &\prec 0\label{eq:thm1_1}
         \\
         P_{h}(0) - P_{q}(T_q) &\prec 0 \label{eq:thm1_2}
    \end{align}
    hold for all $k \in [0,T_h-1]$, $h,q \in \mathcal{I}[1, N(N-1)^{L-1}d^{L}]$, $i_{h,\ell} \in \mathcal{I}[1,N]$ and $\ell \in \mathcal{I}[1,L]$.
    \item [(\emph{b})] There exist symmetric matrices $P_h \in \mathbb{S}_{\succ 0}^{n_x}$, $h \in \mathcal{I}[1, N(N-1)^{L-1}d^{L}]$ such that
    \begin{align}
         \Phi_h^{\top}{P_{h}}\Phi_h-P_{q} \prec 0 \label{eq:thm1_3}
    \end{align}
   where $\Phi_h = \prod_{\ell={0}}^{L-1}A^{\tau_{h,L-\ell}}
        _{i_{h,L-\ell}}$, holds for all $h,q \in \mathcal{I}[1, N(N-1)^{L-1}d^{L}]$.
\end{enumerate}
and if one of the above two statements holds, then switched system (\ref{eq:system}) is GUAS with any switching signals $\sigma(k) \in \mathcal{D}_{[\tau_{\min},\tau_{\max}]}$.
\end{theorem}
\begin{proof}
\textbf{(\emph{a}) $\Rightarrow$ GUAS\emph{:}}  Consider $L$-switching-cycle  $[k_{nL},k_{(n+1)L}]$ with $\mathcal{L}_h$,  $h \in \mathcal{I}[1, N(N-1)^{L-1}d^{L}]$, a Lyapunov function candidate is constructed in the following form
\begin{align} \label{eq:pthm1_1}
 V_{h}(x(k)) = x^{\top}(k)P_{h}(k-k_{nL}) x(k),~k \in [k_{nL},k_{(n+1)L}]
\end{align}
where $P_{h}(k) \succ 0$ for all $k \in \mathcal{I}[0,T_h]$ and  $h \in \mathcal{I}[1, N(N-1)^{L-1}d^{L}]$.

Defining $\Delta V_{h}(k) = V_{h}(x(k+1))-V_{h}(x(k))$ and using (\ref{eq:thm1_1}), the following result
\begin{align} \label{eq:pthm1_2}
    \Delta V_{h}(k) = x^{\top}(k)\left( A_{i_{h,\ell}}^{\top}P_{h}(k+1)A_{i_{h,\ell}}-P_{h}(k)\right) x(k) < 0 
\end{align}
holds for any $k \in [k_{nL},k_{(n+1)L}]$.

Then,  let's consider the concatenation between two individual $L$-switching cycles. Suppose there are two $L$-switching cycles $h,q \in \mathcal{I}[1, N(N-1)^{L-1}d^{L}]$, and assume that system is switching from cycle $q$ to $h$ at time instant $k_{nL}$, one has $V_q(x(k_{nL})) = x^{\top}(k_{nL})P_{q}(T_q)x(k_{nL})$ and $V_h(x(k_{nL})) = x^{\top}(k_{nL})P_{h}(0)x(k_{nL})$. From (\ref{eq:thm1_2}), we have
\begin{align} \label{eq:pthm1_3_0}
    V_h(x(k_{nL})) < V_q(x(k_{nL}))
\end{align}
which implies that the value of Lyapunov function (\ref{eq:pthm1_1}) is also strictly decreasing at the concatenation instant of two $L$-switching cycles. 

Therefore, combining (\ref{eq:pthm1_2}) and (\ref{eq:pthm1_3_0}), Lyapunov function (\ref{eq:pthm1_1}) is strictly decreasing along with the time under the prescribed ranged dwell time switching, the GUAS of system (\ref{eq:system}) can be established by standard Lyapunov theorem \cite{khalil2002nonlinear}. The proof of (\emph{a}) $\Rightarrow$ GUAS is complete.

\textbf{(\emph{a}) $\Rightarrow$ (\emph{b})\emph{:}} Starting from $i_{h,1}$, we  iterate (\ref{eq:thm1_1}),  the  following  inequality can be obtained
\begin{align} \label{eq:pthm1_3}
    (A_{i_{h,1}}^{\tau_{h,1}})^{\top}P_{h}(\tau_{h,1})A_{i_{h,1}}^{\tau_{h,1}} -  P_{h}(0)\prec 0.
\end{align}
Then, for $i_{h,2}$, one has
\begin{align} \label{eq:pthm1_4}
        (A_{i_{h,2}}^{\tau_{h,2}})^{\top}P_{h}(\tau_{h,1}+\tau_{h,2})A_{i_{h,2}}^{\tau_{h,2}} -  P_{h}(\tau_{h,1})\prec 0.
\end{align}
Combining (\ref{eq:pthm1_3}) and (\ref{eq:pthm1_4}), it yields
\begin{align} \label{eq:pthm1_5}
        (A_{i_{h,2}}^{\tau_{h,2}}A_{i_{h,1}}^{\tau_{h,1}} )^{\top}P_{h}(\tau_{h,1}+\tau_{h,2})A_{i_{h,2}}^{\tau_{h,2}}A_{i_{h,1}}^{\tau_{h,1}}  -  P_{h}(0)\prec 0.
\end{align}

Repeating the above process from $\ell=1$ to $\ell=L$, the following inequality holds
\begin{align} \label{eq:pthm1_6}
        \left(\prod_{\ell={0}}^{L-1}A^{\tau_{h,L-\ell}}
        _{i_{h,L-\ell}}\right)^{\top}P_{h}\left(\sum_{\ell=1}^{L}{\tau_{h,\ell}}\right)\prod_{\ell={0}}^{L-1}A^{\tau_{h,L-\ell}}
        _{i_{h,L-\ell}} -  P_{h}(0)\prec 0.
\end{align}

Denoting $\Phi_h = \prod_{\ell={0}}^{L-1}A^{\tau_{h,L-\ell}}
        _{i_{h,L-\ell}}$ and due to $T_h = \sum_{\ell=1}^{L}\tau_\ell$, (\ref{eq:pthm1_6}) can be rewritten to
\begin{align} \label{eq:pthm1_7}
        \Phi_h^{\top}P_{h}(T_h)\Phi_h-  P_{h}(0)\prec 0.
\end{align}

Then, using (\ref{eq:thm1_2}), it leads to
\begin{align} \label{eq:pthm1_8}
            \Phi_h^{\top}P_{h}(T_h)\Phi_h-  P_{q}(T_q)\prec 0.
\end{align}

Letting $P_{h}(T_h) = P_{h}$ and $P_{q}(T_q) = P_{q}$ in (\ref{eq:pthm1_8}), we can conclude that (\ref{eq:thm1_3}) in Statement (\emph{b}) hold.  The proof of (\emph{a}) $\Rightarrow$ (\emph{b}) is complete.

\textbf{(\emph{b}) $\Rightarrow$ (\emph{a})\emph{:}} First, we prove that there always exist $P_{h}(k) \succ 0$ such that (\ref{eq:thm1_1}) holds. 

Consider the following equation 
\begin{align}\label{eq:pthm1_9}
    A_{i_{h,\ell}}^{\top}P_{h}(k+1)A_{i_{h,\ell}}-P_{h}(k) = -Q_{h}(k)
\end{align}
where $Q_{h}(k) \succ 0$ which can be arbitrarily chosen. Given any $P_{h}(T_h)\succ 0$ and staring from $i_{h,L}$, there always exists the solution $P_{h}(k)$ in $[T_h-\tau_{h,L},T_h]$ to (\ref{eq:pthm1_9}) in the form of
\begin{align} \label{eq:pthm1_10}
    P_{h}(k) = (A_{i_{h,L}}^{T_h-k})^{\top}P_{h}(T_h)A_{i_{h,L}}^{T_h-k}+\hat Q_{h}(k)
\end{align}
where 
\begin{align*}
    \hat Q_{h}(k)=\sum_{n=k}^{T_h-1}\left(A_{i_{h,L}}^{(n-k)}\right)^{\top}Q_{h}(n)A_{i_{h,L}}^{(n-k)} \succ  0.
\end{align*}
It is explicitly implies that $P_{h}(k) \succ 0$, $k \in [T_h-\tau_{h,L},T_h]$. We can let $k=T_h-\tau_{h,L}$ in (\ref{eq:pthm1_10}), it arrives at 
\begin{align}\label{eq:pthm1_11}
         P_{h}(T_h-\tau_{h,L}) = (A_{i_{h,L}}^{\tau_{h,L}})^{\top}P_{h}(T_h)A_{i_{h,L}}^{\tau_{h,L}}+\hat Q_{h}(T_h-\tau_{h,L})
\end{align}
then, we can use $P_{h}(T_h-\tau_{h,L})$ to repeat above process to obtain $P_{h}(k) \succ 0$, $k \in [T_h-\tau_{h,L-2},T_h-\tau_{h,L-1}]$. 

Repeating above process from $\ell = L$ to $\ell =  1$, one can obtain $P_{h}(k) \succ 0$, $k \in [0,T_h]$ such that (\ref{eq:pthm1_9}) holds which implies (\ref{eq:thm1_1}) can be established. In particular, letting $k=0$, we can obtain that
\begin{align}\label{eq:pthm1_12}
 P_{h}(0) = \left(\prod_{\ell={0}}^{L-1}A_{i_{h,L-\ell}}^{\tau_{h,L-\ell}}\right)^{\top}P_{h}(T_h)\prod_{\ell={0}}^{L-1}A^{\tau_{h,L-\ell}}
        _{i_{h,L-\ell}} + W_h
\end{align}
where 
\begin{align} \nonumber
    W_h =\sum_{s=1}^{L-1}\left(\prod_{\ell={1}}^{s}A_{i_{h,L-\ell}}^{\tau_{h,L-\ell}}\right)^{\top}\hat Q_{h}(T_h-\tau_{h,L-s})\prod_{\ell={1}}^{s}A^{\tau_{h,L-\ell}}_{h,i_{L-\ell}}
    \\
    + \hat Q_{h,i_1}(0) \nonumber
\end{align}

Then, since (\ref{eq:thm1_3}) holds, it implies that there exist a $\xi>0$ such that
\begin{align} \label{eq:pthm1_13}
    \left(\prod_{\ell={0}}^{L-1}A_{i_{h,L-\ell}}^{\tau_{h,L-\ell}}\right)^{\top}P_{h}\prod_{\ell={0}}^{L-1}A^{\tau_{h,L-\ell}}
        _{i_{h,L-\ell}}-P_{q} \prec -\xi I .
\end{align}

Using (\ref{eq:pthm1_12}) and selecting $P_{h}(T_h)=P_h$, $P_{q}(T_h)=P_q$, we can derive the following inequality
\begin{align} \label{eq:pthm1_24}
    P_{h}(0)-P_{q}(T_q)\prec -\xi I+W_h.
\end{align}

It is noted that $Q_{h}(k)$ in $W_h$ can be arbitrarily chosen so that we are able to adjust $W_h$ sufficiently small to achieve 
\begin{align} \label{eq:pthm1_25}
      P_{h,i_1}(0)-P_{q,i_L}(\tau)\prec -\xi I
\end{align}
Therefore, (\ref{eq:thm1_2}) holds. The proof of (\emph{b}) $\Rightarrow$ (\emph{a}) is complete. 
\end{proof}

\begin{remark}\label{remark2}
Theorem \ref{thm1} provides two sufficient conditions for stability of discrete-time switched linear system (\ref{eq:system}). The following observations can be made for these two conditions: 
\begin{enumerate}
    \item [1)] The stability analysis results are closely related to $L$. A larger $L$ will lead to a less conservative computation result. Notably, a sufficiently large $L$ is expected to even reach nonconservative stability analysis results. Moreover, Theorem \ref{thm1} can be viewed as a generalization of several well-known stability criteria for discrete-time switched linear system (\ref{eq:system}) which as special cases with different $L$ in Theorem \ref{thm1}. The conservatism and comparison with existing results will be further discussed in detail in the following subsections. 
    \item [2)] Condition (\emph{a}) is convex in system matrices $A_i$, thus (\ref{eq:thm1_1})--(\ref{eq:thm1_2}) in Condition (\emph{a}) can be easily extended to other related problems such as robust stability analysis for uncertain switched systems, $\ell_2$-gain analysis, and controller design. 
    \item [3)] Condition (\emph{b}) involves the intricate multiplication of $A^{\tau}_{i_{h,\ell}}$, i.e., $\prod_{\ell={0}}^{L-1}A^{\tau}
        _{i_{h,L-\ell}}$ which makes further extensions not as convenient as Condition (\emph{a}), but the computational complexity is significantly reduced compared with Condition (\emph{a}). 
        A  computation complexity comparison is listed in Table \ref{tab1}. 
\end{enumerate}

\begin{table}
\caption{Computational Complexity of Conditions (\emph{a}) and (\emph{b}) where $M=N(N-1)^{(L-1)}d^{L}$}\label{tab1}
\begin{center}
\renewcommand{\arraystretch}{1.4}
\begin{tabular}{|c|c|c|}
\hline
 \textbf{Conditions} & \textbf{Number of variables $P_h$} & \textbf{Number of LMIs} \\ \hline
 (\emph{a}) & $\sum\nolimits_{h=1}^{M} (T_h+1)$ &  $M^2+\sum\nolimits_{h=1}^{M}(T_h-1)$  \\ \hline
 (\emph{b}) & $M$ &  $M^2$\\ \hline
\end{tabular}
\end{center}
\end{table}
\end{remark}

\subsection{Discussion on Conservatism}

As discussed in Remark \ref{remark2}, the stability analysis results are closely related to $L$. Roughly speaking, larger $L$ introduces more decision variables that lead to more relaxed linear matrix inequality (LMI) conditions, so that less conservative results can be achieved. A question naturally arises here: \emph{Can we achieve nonconservative results in the framework of $L$-switching-cycle approach if a sufficiently large $L$ is chosen?} The following theorem provides an affirmative answer. 

\begin{theorem}\label{thm2}
Given a ranged dwell time $[\tau_{\min},\tau_{\max}]$, switched linear system (\ref{eq:system}) is GUAS with any switching signals $\sigma(k)\in\mathcal{D}_{[\tau_{\min},\tau_{\max}]}$ {if and only if} there exist an $L$ such that (\ref{eq:thm1_1})--(\ref{eq:thm1_2}) in Statement (\emph{a}) or (\ref{eq:thm1_3}) in Statement (\emph{b}) in Theorem \ref{thm1} hold. 
\end{theorem}
\begin{proof}
The sufficiency, namely (\ref{eq:thm1_1})--(\ref{eq:thm1_2}) or (\ref{eq:thm1_3}) $\Rightarrow$  GUAS, has been proved in Theorem 1. Here we only need to consider the necessity, that is GUAS $\Rightarrow$ (\ref{eq:thm1_1})--(\ref{eq:thm1_3}) or (\ref{eq:thm1_3}). 

Since system (\ref{eq:system}) is GUAS under with any switching signals $\sigma(k)\in\mathcal{D}_{[\tau_{\min},\tau_{\max}]}$, we consider any $L$-switching-cycle indexed by $h$. Using the definition of state transition matrix $x(k) = \Phi(k,0)x(0)$, we have $x(T_h) = \Phi_h(T_h,0)x(0)$ with 
\begin{align} \label{eq:pthm1_24}
    \Phi_h(T_h,0) = \prod_{\ell={0}}^{L-1}A^{\tau_{h,L-\ell}}
        _{i_{h,L-\ell}}
\end{align}
for an $L$-switching-cycle. Moreover, due to $T_h = \sum\nolimits_{\ell=1}^{L}\tau_{h,\ell}$, we have 
\begin{align} \label{eq:pthm1_24}
    \Phi_h\left(\sum_{\ell=1}^{L}\tau_{h,\ell},0\right) = \prod_{\ell={0}}^{L-1}A^{\tau_{h,L-\ell}}
        _{i_{h,L-\ell}} .
\end{align}

Since system (\ref{eq:thm1_1}) is GUAS, there exists a $\mathcal{KL}$ function $\beta$ such that $\left\|x(T_h)\right\| \le \beta(\left\|x(0)\right\|,T_h)$ which implies $\left\|\Phi(T_h,0)x(0) \right\|\le \beta(\left\|x(0)\right\|,T_h)$. Then, since $\beta$ is a $\mathcal{KL}$ function, it implies $\lim_{T_h \to \infty} \beta(\left\|x(0)\right\|,T_h) = 0$. Furthermore, due to $T_h = \sum\nolimits_{\ell=1}^{L}\tau_{h,\ell}$ where $1 \le \tau_{\min}\le \tau_{h,\ell} \le \tau_{\max} < \infty$,  it is equivalent to $\lim_{L \to \infty} \beta(\left\|x(0)\right\|,\sum\nolimits_{\ell=1}^{L}\tau_{h,\ell}) = 0$. As a result, it leads to $\lim_{L \to \infty}\left\|\Phi(\sum\nolimits_{\ell=1}^{L}\tau_{h,\ell},0)x(0) \right\| = 0$ which leads to $\lim_{L \to \infty}\left\|\Phi(\sum\nolimits_{\ell=1}^{L}\tau_{h,\ell},0)\right\| = 0$.  Therefore, for any $P_h,P_q \succ 0$, we can have 
\begin{align}
    \lim_{L \to \infty} \Phi^{\top}\left(\sum_{\ell=1}^{L}\tau_{h,\ell},0\right)P_h\Phi\left(\sum_{\ell=1}^{L}\tau_{h,\ell},0\right) - P_q  
    = -P_q \prec 0
\end{align}
which means that there exists a sufficiently large $L^*$ such that the following inequality holds
\begin{align}
    \Phi^{\top}\left(\sum_{\ell=1}^{L}\tau_{h,\ell},0\right)P_h\Phi\left(\sum_{\ell=1}^{L}\tau_{h,\ell},0\right) - P_q \prec 0 
\end{align}
for any $L \ge L^*$.

From (\ref{eq:pthm1_24}), we can obtain 
\begin{align}
            \left(\prod_{\ell={0}}^{L-1}A^{\tau_{h,L-\ell}}_{i_{h,L-\ell}}\right)^{\top}{P_h}\prod_{\ell={0}}^{L-1}A^{\tau_{h,L-\ell}}
        _{i_{h,L-\ell}}-P_{q} \prec 0
\end{align}
and thus (\ref{eq:thm1_3}) holds by letting $\Phi_h = \prod_{\ell={0}}^{L-1}A^{\tau_{h,L-\ell}}
        _{i_{h,L-\ell}}$. The proof of GUAS $\Rightarrow$ (\ref{eq:thm1_3}) is complete. Furthermore, since (\ref{eq:thm1_1})--(\ref{eq:thm1_2}) is equivalent to (\ref{eq:thm1_3}) by Theorem \ref{thm1}, we can also conclude that GUAS $\Rightarrow$ (\ref{eq:thm1_1})--(\ref{eq:thm1_2}). The proof is complete.
\end{proof}

Theorem \ref{thm2} explicitly states that the developed $L$-switching-cycle method is not only able to provide sufficient stability conditions, but also can achieve nonconservativeness as long as a sufficiently large $L$ is provided. 

\subsection{Comparison with Existing Results}
Theorems \ref{thm1} and \ref{thm2} covers several existing stability analysis results for discrete-time switched linear systems. By particularly selecting $L$ and considering special dwell time, e.g., $\tau_{\max}=\tau_{\min}=1$ implies arbitrary switching, Theorems \ref{thm1} and \ref{thm2} can be reduced to those existing results. 

Letting $L=1$, we have $T_h \in [\tau_{\min},\tau_{\max}]$, and thus (\ref{eq:thm1_1})--(\ref{eq:thm1_2}) in Condition (\emph{a}) are reduced to (\ref{eq:lemma1_1})--(\ref{eq:lemma1_2}) in Lemma \ref{lemma1}. In addition, 
(\ref{eq:thm1_3}) in Condition (\emph{b})  are reduced to (\ref{eq:lemma2_1}) in Lemma \ref{lemma2}. 
Therefore, Lemma \ref{lemma1} and Lemma \ref{lemma2} can be viewed as special cases of the developed $L$-switching-cycle method with $L=1$. The results in Lemma \ref{lemma1} and Lemma \ref{lemma2} are both sufficient conditions. In this paper, by introducing the concept of $L$-switching-cycle, the obtained stability analysis results will be less conservative and even be able to reach non-conservativeness as shown in Theorem \ref{thm2}. 

Theorems \ref{thm1} and \ref{thm2} also cover results regarding arbitrary switching. In \cite{xiang2017robust}, a multiple-step method is proposed to deal with stability analysis for discrete-time linear system under arbitrary switching.  If we enforce $\tau_{\max} = \tau_{\min}=1$ and allow $i_{h,\ell} = i_{h,\ell+1} $ in Statements (\emph{a}) and (\emph{b}) in Theorem \ref{thm1}, the $L$-switching-cycle method is reduced to the multiple-step method, e.g., (\ref{eq:thm1_1})--(\ref{eq:thm1_2}) become
\begin{align} \label{eq:14_1}
         A_{i_{h,\ell}}^{\top}P_{h}(k+1)A_{i_{h,\ell}}-P_{h}(k) &\prec 0
         \\
         P_{h}(0) - P_{q}(L) &\prec 0  \label{eq:14_2}
    \end{align}
where $h,q \in \mathcal{I}[1, N^{L}]$. Moreover, (\ref{eq:thm1_3}) is reduced to the form of
\begin{align} \label{eq:14_3}
     \left(\prod_{\ell={0}}^{L-1}A
        _{i_{h,L-\ell}}\right)^{\top}{P_h}\prod_{\ell={0}}^{L-1}A
        _{i_{h,L-\ell}}-P_{q} \prec 0
\end{align}
where $h,q \in \mathcal{I}[1, N^{L}]$.
It is noted that conditions (\ref{eq:14_1})--(\ref{eq:14_3}) are exactly same as in \cite{xiang2017robust}. Furthermore, as proved in \cite{xiang2017robust}, the non-conservativeness of stability analysis results can be obtained as long as the number of steps is sufficiently large. Theorem \ref{thm2} can obtain the same nonconservative result as long as $L$ is large enough. Moreover, if we further enforce $L=1$ along with  $\tau_{\max} = \tau_{\min}=1$, the results can be relaxed to 
\begin{align}
    A_i^{\top}P_iA_i - P_j \prec 0, ~i,j \in \mathcal{I}[1,N]
\end{align}
which is the classical switched Lyapunov function approach proposed in \cite{daafouz2002stability}. Therefore, the multiple-step method in \cite{xiang2017robust} and switched Lyapunov function approach in \cite{daafouz2002stability} are also special cases of $L$-switching-cycle method with $\tau_{\max} = \tau_{\min}=1$ and further with $L=1$, respectively.

\section{Example}
In this example, we consider system (\ref{eq:system}) with two subsystems as below:
\begin{align}
    A_1 = \begin{bmatrix}
    1  & 0.1
    \\ 
    -0.2 & 0.9
    \end{bmatrix},~    A_2 = \begin{bmatrix}
    1 & 0.1
    \\
    -0.9 & 0.9
    \end{bmatrix}
\end{align}
For the sake of simplicity, we consider periodic switching, i.e., $k_{n+1}-k_{n} = \tau$, $\forall n \in \mathbb{N}$.  
It has been demonstrated in \cite{dehghan2013computations} that this system is GUAS under any switching with minimum dwell time $\tau_{\min} \ge 15$. Thus, we can conclude that the system is GUAS under periodic switching for all cycles with dwell time $\tau \ge 15$.  

In this example, we will further explore the stability of this system under periodic  switching with dwell time $\tau \le 15$. We employ $L$-switching-cycle method with $L \in \mathcal{I}[1,10]$ to examine stability with dwell time $\tau \in \mathcal{I}[1,15]$. The computation results of admissible dwell time $\tau$ are elucidated in Fig. \ref{fig_2} and Table \ref{tab2}.  It can be observed that $L$-switching-cycle method can lead to a less conservative stability analysis results for most of cases, i.e., Lemmas \ref{lemma1} and \ref{lemma2} fail for stability analysis of $\tau \in \{ 1,2,3,4,5,9,12,13,15\}$. Moreover, the results in Table \ref{tab2} are nonconservative stability results for all $\tau \in \mathcal{I}[1,15]$ since $A_1^{\tau}A_2^{\tau}$, $\tau \in \{  6,7,8,14\}$, contains eigenvalues outside the unit disc which directly leads to instability of the system under periodic switching with dwell time $\tau\in \{  6,7,8,14\}$. 

To validate the theoretical analysis results by the $L$-switching-cycle method, the state responses are presented in Fig. \ref{fig_1}. It can be observed that systems trajectories diverge when $\tau \in \{6,7,8,14\}$ while trajectories converge in the rest of cases, which is consistent with Fig. \ref{fig_1} and Table \ref{tab2}.

\begin{figure}
\centering
	\includegraphics[width=9cm]{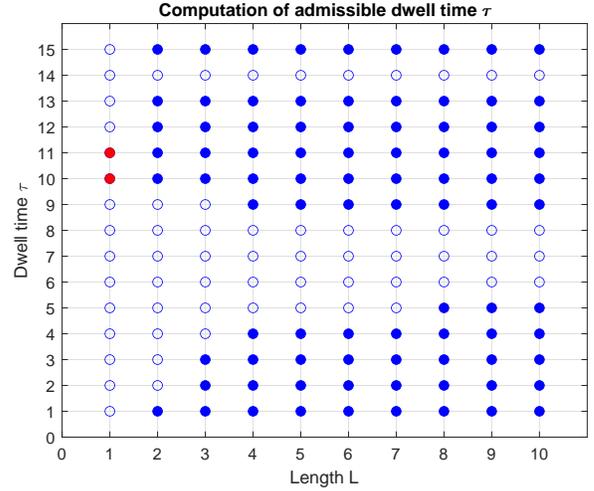}
	\caption{Computation results of admissible dwell time $\tau$ with different length $L$. Blue and red solid circles are admissible dwell time with respect to the given $L$. Red solid circles means that the results can be also obtained by Lemmas \ref{lemma1} and \ref{lemma2}.}
	\label{fig_2} 
\end{figure}

\begin{table}
\caption{Stability Analysis Results with Different Dwell Time $\tau$}\label{tab2}
\begin{center}
\renewcommand{\arraystretch}{1.2}
\begin{tabular}{|c|c|c|c|}
\hline
 \textbf{Dwell Time} & \textbf{Stability Results} &\textbf{Length $L$} & \textbf{Lemmas \ref{lemma1}, \ref{lemma2}} \\ \hline
$\tau = 1$ & Stable  &  $L \ge 2$ &  No\\ \hline
 $\tau = 2$ & Stable  &  $L \ge 3$ &  No\\ \hline
 $\tau = 3$ & Stable  &  $L \ge 3$  & No\\ \hline
 $\tau = 4$ & Stable  &  $L \ge 4$  & No\\ \hline
 $\tau = 5$ & Stable  &  $L \ge 8$  & No\\ \hline
 $\tau = 6$ & Unstable  &  --  & --  \\ \hline
 $\tau = 7$ & Unstable  &  --  &  -- \\\hline
 $\tau = 8$ & Unstable  &  -- & --\\ \hline
 $\tau = 9$ & Stable  &  $L \ge 4$  & No\\ \hline
 $\tau = 10$ & Stable  &  $L \ge 1$  &  Yes\\ \hline
 $\tau = 11$ & Stable  &  $L \ge 1$  & Yes \\ \hline
 $\tau = 12$ & Stable  &  $L \ge 2$   & No\\ \hline
 $\tau = 13$ & Stable  &  $L \ge 2$  & No\\ \hline
 $\tau = 14$ &  Unstable   & -- & -- \\ \hline
 $\tau = 15$ & Stable  &  $L \ge 2$  & No\\ \hline
\end{tabular}
\end{center}
\end{table}

\begin{figure*}
\centering
	\includegraphics[width=17.5cm]{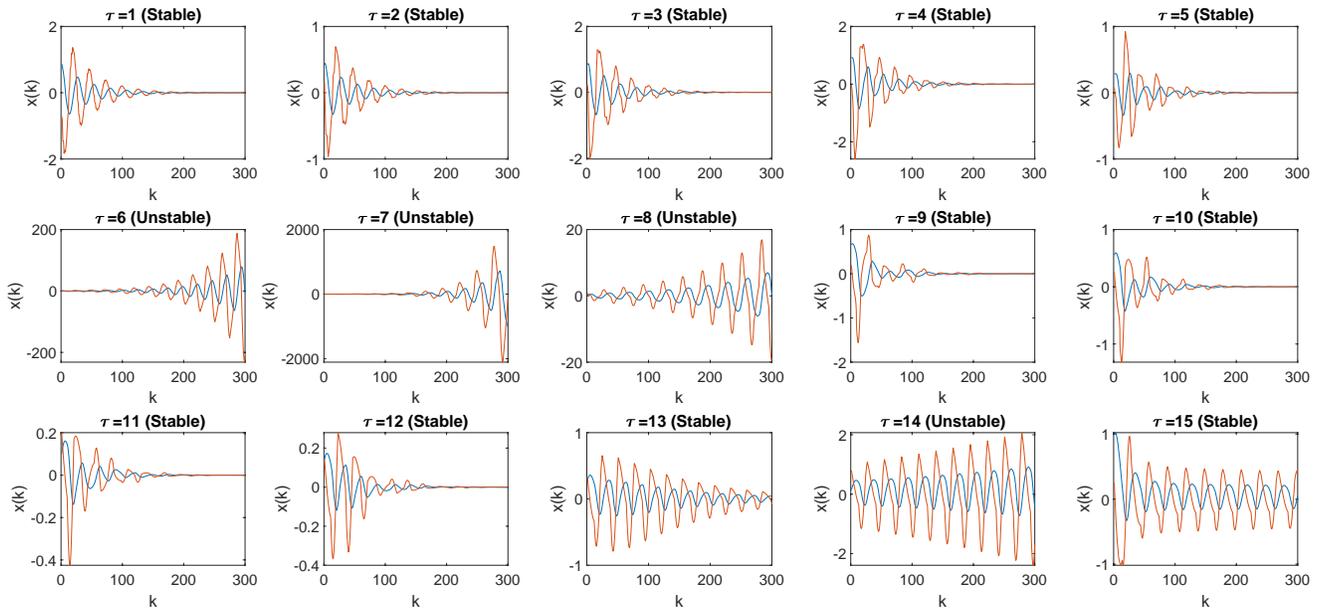}
	\caption{State responses under periodic switching with $\tau \in \mathcal{I}[1,15]$. It can be observed that systems trajectories diverges (unstable) when $\tau \in \{6,7,8,14\}$ while trajectories converges (stable) when $\tau \in \{ 1,2,3,4,5,9,10,11,12,13,15\}$, which is consistent with Fig. \ref{fig_1} and Table \ref{tab2}.}
	\label{fig_1} 
\end{figure*}

\section{Conclusions}
This paper proposes a novel concept called $L$-switching-cycle to perform stability analysis for discrete-time switched linear systems. Two sufficient stability conditions are developed based on $L$-switching-cycle, and it has been proved that two conditions are essentially equivalent in stability analysis results. Two stability results can be viewed as improvements over results based on multiple Lyapunov functions and clock-dependent Lyapunov function methods. Furthermore, if a sufficiently large $L$ is employed, two stability conditions can eventually achieve nonconservative stability analysis results. It also can be observed that $L$-switching-cycle stability criteria cover a variety of existing stability criteria.  Based on the results proposed in this paper, the extensions of $L$-switching-cycle to related problems such as controller design, input-output performance analysis needs further investigations. In addition, how to reduce the computational complexity and
make it applicable for switched systems with a large number of subsystems also needs further studies.

\bibliographystyle{ieeetr}
\bibliography{ref}

\end{document}